\documentclass{amsart}
\usepackage[backend=biber, doi=true, url=false]{biblatex}
\DeclareFieldFormat{doi}{\href{https://doi.org/#1}{DOI}}
\DeclareFieldFormat{eprint:arxiv}{arXiv: \href{https://arxiv.org/abs/#1}{#1}}
\DeclareFieldFormat{issn}{}
\usepackage[colorlinks=True, citecolor=blue, linkcolor=blue,hypertexnames=false]{hyperref}
\usepackage{amssymb}
\usepackage{tikz}
\usepackage{tikz-cd}
\usepackage{tkz-graph}
\usepackage{float}
\usepackage{faktor}
\usetikzlibrary{shapes.geometric}
\usepackage{algorithm}
\usepackage{algorithmic}
\usepackage{mathtools}
\usepackage[nameinlink]{cleveref}
\usepackage{booktabs}
\usepackage{longtable}
\usepackage{placeins}
\usepackage[shortlabels,inline]{enumitem}
\usepackage[normalem]{ulem}
\usepackage{comment}
\usepackage{thmtools}
\declaretheorem[name=Theorem,style=plain,parent=section]{theorem}
\declaretheorem[name=Proposition,style=plain,numberlike=theorem]{proposition}
\declaretheorem[name=Lemma,style=plain,numberlike=theorem]{lemma}
\declaretheorem[name=Corollary,style=plain,numberlike=theorem]{corollary}
\declaretheorem[name=Definition,style=definition,numberlike=theorem]{definition}
\declaretheorem[name=Remark,style=definition,numberlike=theorem]{remark}
\declaretheorem[name=Question,style=definition,numberlike=theorem]{question}
\DeclareMathOperator{\Nef}{Nef}
\DeclareMathOperator{\Eff}{Eff}
\DeclareMathOperator{\Cl}{Cl}
\DeclareMathOperator{\Pic}{Pic}
\DeclareMathOperator{\id}{id}
\DeclareMathOperator{\Hom}{Hom}
\DeclareMathOperator{\Aut}{Aut}
\DeclareMathOperator{\Autstar}{Aut^{\!*\!}}
\DeclareMathOperator{\Spec}{Spec}
\DeclareMathOperator{\Div}{Div}
\DeclareMathOperator{\divi}{div}
\newcommand{\QQ}{\mathbb{Q}}
\newcommand{\FF}{\mathbb{F}}
\newcommand{\RR}{\mathbb{R}}
\newcommand{\ZZ}{\mathbb{Z}}
\newcommand{\PP}{\mathbb{P}}
\newcommand{\HH}{\mathbb{H}}
\renewcommand{\O}{\mathcal{O}}
\renewcommand{\P}{\mathcal{P}}
\DeclareMathOperator{\Num}{Num}

\makeatletter
\newcommand*{\defeq}{\mathrel{\rlap{\raisebox{0.3ex}{$\m@th\cdot$}}\raisebox{-0.3ex}{$\m@th\cdot$}}=}
\makeatother
\title[The Cone Conjecture for Enriques Surfaces in any Characteristic]{The Cone Conjecture for Enriques\\Surfaces in any Characteristic}
\author{Simon Brandhorst}
\address{Fakultät für Mathematik und Informatik \\ Universität des Saarlandes \\ Campus\newline Saarbrücken \\ Gebäude E2 4 \\ 66123 Saarbrücken \\ Germany}
\email{brandhorst@math.uni-sb.de}
\author{Gebhard Martin}
\address{Mathematisches Institut \\ Universität Bonn \\ Endenicher Allee 60 \\ 53115 Bonn,\newline Germany}
\email{gmartin@math.uni-bonn.de}
\author{Tobias Schnieders}
\address{Fakultät für Mathematik und Informatik \\ Universität des Saarlandes \\ Campus\newline Saarbrücken \\ Gebäude E2 4 \\ 66123 Saarbrücken \\ Germany}
\email{schnieders@math.uni-sb.de}
\begin{document}
\begin{abstract}
We give a proof of the Morrison--Kawamata cone conjecture for Enriques surfaces independent of their characteristic. It is based on the analysis of certain generically finite morphisms of degree two.
\end{abstract}
\maketitle
\section{Introduction}
Let $X$ be a normal projective variety over an algebraically closed field and with numerically trivial canonical bundle. The Morrison--Kawamata cone conjecture predicts that the effective nef cone of $X$ is rational polyhedral up to the action of $\Aut(X)$. It is a finiteness statement with far reaching consequences. Over the complex numbers, the Morrison--Kawamata cone conjecture is known for K3 and Enriques surfaces by Sterk \cite{sterk85} and Namikawa \cite{namikawa85}, for abelian varieties (in arbitrary characteristic) by Prendergast-Smith \cite{Pre12}, for holomorphic symplectic varieties (under a mild condition) by Amerik--Verbitsky, Lehn--Mongardi--Pacienza \cite{amerik-verbitsky2017, lehn_mongardi_pacienza}, for a few Calabi--Yau varieties, and finite quotients of the aforementioned by Gachet \cite{gachet2025}. In almost all cases, the central pillar of the proof is a Torelli-type theorem. 

In positive characteristic, the Morrison--Kawamata cone conjecture is known for most surfaces. The exceptions were K3 and Enriques surfaces in characteristic $2$. 
The proof due to Lieblich--Maulik \cite{maulik} in the K3 case in odd characteristic comes in two flavors: Either the K3 surface $X$ and its automorphism group can be lifted (up to finite index) to characteristic zero and the Hodge theoretic Torelli-type theorem is invoked, or $X$ is supersingular and there is the crystalline Torelli theorem \cite{ogus:83,bragg-lieblich}(which is still open in some cases in characteristic $2$). 
Since an Enriques surface $S$ in odd characteristic is covered by a K3 surface $X$, one can descend the conjecture from $X$ to $S$, as has been carried out in work of Wang \cite{wang-cone}. Characteristic $2$ requires a new approach.

We propose a new proof of the Morrison--Kawamata cone conjecture for Enriques surfaces which does not rely on a Torelli-type theorem and instead uses the geometry of double covers. This has the merit of working independently of the characteristic and in particular in characteristic $2$, where it was previously not known.
\begin{theorem} \label{thm: main}
Let $S$ be an Enriques surface over a perfect field $k$.
Then $\Aut(S)$ has a rational polyhedral fundamental domain on the effective nef cone $\Nef^e(S)=\Nef(S)\cap \Eff(S)$. 
\end{theorem}
\begin{corollary}\label{coro:of-main-thm}
The group $\Aut(S)$ is finitely presented.
\end{corollary}
See \Cref{coro:geometrically-finite,coro:527} for further implications on the geometry of Enriques surfaces. 
\begin{remark}\label{rem:perfect-field}
By \cite[Theorem 1.6]{gachet2025}, we may and will assume that $k$ is algebraically closed for the remainder of the article. Note that we can also drop the assumption that $k$ is perfect if $H^0(S,T_S) = H^1(S,\mathcal{O}_S) = 0$ (which holds if ${\rm char}(k) \neq 2$ or if ${\rm char}(k) = 2$ and $S$ is classical and non-exceptional \cite[Theorem 1.4.10]{cdl:enriquesI}), for then all objects involved in the statement of \Cref{thm: main} remain unchanged when passing from a separable to an algebraic closure of $k$.
\end{remark}
In some respects, Enriques surfaces behave similarly to Calabi--Yau varieties: Their Picard number is constant in families. Therefore, we hope that our proof of the cone conjecture could pave the way to new proofs for some Calabi--Yau varieties, where we do not have a Torelli-type theorem but maybe enough realizations as a Galois cover of other varieties. 
\subsection*{Outline and notation}
Before we sketch the idea of the proof we introduce the necessary notation. 
\begin{itemize}
\item We work over an algebraically closed field $k=\bar k$. 
\item  $S/k$ denotes an \emph{Enriques surface}, that is, a smooth proper surface $S/k$ such that its second étale Betti number $b_2(S)=10$ and its canonical bundle is numerically trivial. 
\item The \emph{numerical lattice} of $S$ is denoted by $\Num(S)$ and it is isometric to the \emph{Vinberg lattice} $E_{10}\cong U \oplus E_8$ because it is even, unimodular, hyperbolic and of rank $10$.
\item For a convex cone $C$ in a real vector space with a $\QQ$-structure, we denote by $C^+$ the cone spanned by the rational points of $\overline{C}$. 
\item Inside $\Num(S) \otimes \mathbb{R}$, we have the \emph{positive cone} $\P_S^+$, the \emph{nef cone} $\Nef(S)$, and the \emph{effective cone} $\Eff(S)$, spanned by positive, nef, and effective classes, respectively.
\item The \emph{effective nef cone of $S$} is $\Nef^e(S) = \Nef(S) \cap \Eff(S)$. Since the canonical bundle of $S$ is numerically trivial and $\chi(\mathcal{O}_S) > 0$, $\Nef^e(S)=\Nef^+(S)$ follows from Riemann--Roch. 
\item The image of $\Aut(S) \to O(\Num(S))$ is denoted by $\Autstar(S)$. Its kernel is the subgroup of \emph{numerically trivial automorphisms}, which is a finite group by \cite[Proposition 2.1]{dolgachev_martin:numericallytrivial}.
\item The \emph{nodal Weyl group} of $S$ is the subgroup of $O(\Num(S))$ generated by the reflections $x\mapsto x+\langle r, x\rangle r$ in classes $r$ of $(-2)$-curves of $S$ and is denoted by $W(S)$.
\item The \emph{Weyl automorphism group} is the inner semidirect product
$G_S\defeq W(S)\rtimes \Autstar(S) \subseteq O(\Num(S)).$
\end{itemize}
Now we come to the proof.
As a first reduction, we use $W(S)$ to replace $\Nef^e(S)$ by $\P_S^+$ and $\Autstar(S)$ by $G_S$. By results of Vinberg \cite{vinberg1972}, $O^+(E_{10})$ acts with a rational polyhedral fundamental domain on $\P_S^+$. Therefore, it suffices to show that $G_S$ is of finite index in $O^+(\Num(S))$. We show that $G_S$ contains the $2$-congruence subgroup $O^+(\Num(S))(2)$ of $O(\Num(S))$. 

By results of Coble and Allcock, the $2$-congruence group is generated by (the infinitely many) involutions of the form $\id_U \oplus -\id_{E_8}$ for any splitting $E_{10}\cong U\oplus E_8$. 

First, suppose that $S$ does not contain any smooth rational curve, that is, $S$ is unnodal (equivalently $\Nef(S)^+=\P_S^+$ is round). 
Then, $U$ induces a finite $2$ to $1$ morphism 
$\phi\colon S \to D$ to an anti-canonical del-Pezzo surface of Picard number $2$ and $\phi^*\Num(D)=U$. The covering involution $\tau$ is called a bielliptic involution. It fixes $U$ and acts as $-\id$ on $U^\perp \cong E_8$. This proves the cone conjecture for unnodal Enriques surfaces and is essentially known, see \cite{martin:automorphisms_of_unnodal_enriques_surfaces}.

The interesting part is what happens when $S$ is nodal. 
Then $\phi$ is only generically finite and $\tau$ may act in a non-trivial way on the collection of exceptional curves. The analysis of this action is at the core of the proof. It turns out that, up to the action of $W(S)$, $\tau^*$ still acts as $\id_U \oplus -\id_{U^\perp}$.
Since we replaced $\Autstar(S)$ by $G_S$, this is enough to make the proof work.
\subsection*{Acknowledgement}
The authors would like to thank Davide Veniani for discussions and C\'ecile Gachet and Kohei Kikuta for their comments on a first version of this article. We thank the referee for their comments. This work is supported by SFB
195 No. 286237555 of DFG. This publication was supported by the Open Access Publication Fund of the University of Bonn.
\section{Preparations}
\subsection{Bielliptic linear systems}
We define $U$-pairs following \cite[Definition 3.3.3]{cdl:enriquesI}.
\begin{definition}
A pair of genus one pencils $|2F_1|$ and $|2F_2|$ on an Enriques surface $S$ with $F_1.F_2=1$ (respectively a genus one pencil $|2F|$ and a smooth rational curve $R$ with $R.F=1$) is called a $U$-pair (respectively degenerate $U$-pair). 
\end{definition}
The reason for this terminology is that the classes of the so-called half-fibres $f_1\defeq[F_1]$ and $f_2\defeq[F_2] \in \Num(S)$ (respectively $f_2\defeq[F+R]$) define a hyperbolic plane $U=\langle f_1,f_2\rangle \hookrightarrow \Num(S)\cong E_{10}$.
The corresponding linear system $|L|=|2F_1+2F_2|$ (respectively $|4F+2R|$) is called non-special (respectively special) \emph{bielliptic} of square $L^2=8$. It is base-point free and defines a generically finite morphism $\phi_{|L|}\colon S\to D\subseteq \PP^4$ of degree $2$ onto an anti-canonical del Pezzo surface $D$ of degree $4$. The curves orthogonal to $L$ are contracted. For more details, see \cite[Chapter 3.3]{cdl:enriquesI}.
\begin{remark}\label{rem:picard_number_D}
Explicitly, in characteristic not $2$, $D=V(x_0^2+x_1x_2,x_0^2+x_3x_4)$ in the non-special case and
$D=V(x_0^2+x_1x_2,x_0x_3+x_4^2)$ in the special case. 
This surface has $4A_1$ (respectively $2A_1+A_3$) singularities.
In characteristic $2$, there are four further types with a $D_4$ or $D_5$ singularity \cite[Corollary 0.6.14, Theorem 3.3.4]{cdl:enriquesI}. In any case, the Picard number of $D$ is $1$ in the special and $2$ in the non-special case. 
\end{remark}
If the double cover $\phi_{|L|}$ is separable, then it is Galois, and we obtain a covering involution $\tau$, which is a morphism since $K_S$ is nef. It is called a \emph{bielliptic involution}. We will describe the action of $\tau$ on $\Num(S)$ in the following two subsections.
\subsection{The local class group of a canonical surface singularity}
Let $P \in X$ be a normal surface singularity over the algebraically closed field $k$. Let $R = \widehat{\O}_{X,P}$ be the completion of the stalk $\O_{X,P}$ and set $U\vcentcolon=\Spec R \setminus \{P\}$.
The \emph{local class group} of $X$ at $P$ is defined as 
\[\Cl_P(X)\defeq \Pic(U) = \Cl(R),\]
where by $\Cl(R)$ we mean the group of formal linear combinations of height $1$ prime ideals modulo principal ideals.

For the rest of this subsection, we pass to the local situation $X = \Spec R$.
Let $\phi \colon \widetilde{X} \to X$  be the minimal resolution and 
\[\phi^{-1}(\{P\}) =\bigcup_{i=1}^n E_i\] 
be its exceptional locus where $E_1, \dots, E_n$ denote the exceptional divisors.
Denote by $E$ the subgroup of $\Div(\tilde{X})$ consisting of divisors with exceptional support. It is spanned by the exceptional divisors $E_1,\dots, E_n$. Its dual module is denoted by $E^*=\Hom_\ZZ(E,\ZZ)$. Viewing $E$ as a quadratic lattice, we can identify the dual module $E^*$ with the dual lattice $E^\vee \subseteq L\otimes \QQ$.

The inclusion $U \subseteq \widetilde{X}$ gives a short exact sequence 
\[0 \to E \to \Pic(\widetilde{X})\to \Pic(U)=\Cl_P(X) \to 0.\]

\begin{proposition}[{\cite[Proposition 17.1]{lipman69}}]\label{class-group_and_discriminant-group}
Suppose that $P$ is a rational singularity and $X=\Spec R$. Then the map 
\[\theta\colon \Pic(\widetilde{X}) \to E^*\cong E^\vee, \quad [D] \mapsto \langle [D], \cdot \;\rangle\] 
descends to an isomorphism 
\[\Cl_P(X) \cong E^\vee /E.\] 
\end{proposition}
\subsection{Class groups and double covers}
In this section we analyze how the Galois group of a finite Galois cover acts on the local class group. 
\begin{proposition}\label{separable-galois-cover}
Let $\phi \colon X \to Y$ be a finite Galois cover of normal varieties with covering group $G$.
Let $P \in X^G$, and assume that $Y$ is smooth at $\phi(P)$. 
Then, $G$ acts on $\Cl_P(X)$, and for all $[D] \in \Cl_P(X)$, we have
\[\sum_{\sigma \in G} \sigma_*[D]=0 \in \Cl_P(X).\]
In particular, for $\deg \phi = 2$, the unique element $\sigma\in\Aut(X/Y)\setminus\{\id\}$ acts as $-1$ on $\Cl_P(X)$.
\end{proposition}
\begin{proof}
Note that $\deg(\phi)=\sharp G$. Let $D$ be a divisor lying in the equivalence class of $[D]$. 
The sum $\sum_{\sigma \in G} \sigma_*D$ is invariant under the action of $G$ and it follows
\[\phi_*\left(\sum_{\sigma \in G} \sigma_*D\right)=\sum_{\sigma\in G}\phi_*D=\deg(\phi)\phi_*D.\]
Since $\Cl_{\phi(P)}(Y)$ is trivial, there exists an $f\in\widehat{\O}_{Y,\phi(P)}$ such that $\divi(f)=\phi_*D$. We conclude
\[\phi_*\left(\sum_{\sigma \in G} \sigma_*D\right)=\deg(\phi)\divi(f)=\divi(N(\phi^*f))=\phi_*(\divi(\phi^*f)),\]
where $N$ denotes the norm of the field extension $K(X)/K(Y)$. In order to finish the proof, we show that $\phi_*$ restricted to $G$-invariant divisors is injective, as both $\sum_{\sigma \in G} \sigma_*D$ and $\divi(\phi^*f)$ are $G$-invariant. So, let $C=\sum_{i \in I} c_i C_i$ (with $C_i$ prime) be a $G$-invariant divisor with $\phi_*C=0$. 
With $J_i\vcentcolon=\{j \in I \mathop\vert \exists\sigma\in G\colon C_j=\sigma_*(C_i)\}$ we find
\[c_i  \cdot \sharp (G C_i)=\sum_{j\in J_i} c_j=0\]
for each $i\in I$, and it follows $C=0$.
\end{proof}
We shall also need the following well-known inseparable case of the above proposition:
\begin{proposition}\label{prop:local-class-inseparable}
Let $\phi \colon X \to Y$ be a finite, purely inseparable morphism of degree $p$ between normal varieties. Then, for every point $P \in X$ such that $\phi(P)$ is smooth, the local class group at $P$ is $p$-torsion:
\[p \Cl_P(X)=0.\] 
\end{proposition}
\begin{proof}
Without loss of generality, we may assume that $Y = \Spec A$ and $X = \Spec B$ where $A$ is a localization of a finitely generated smooth $k$-algebra, $B$ is a normal local ring, which is the integral closure of $A$ in a purely inseparable degree $p$ extension $L/K$, where $K$ is the fraction field of $A$. In particular, it holds that  ${\rm char}(K) = p > 0$. Consider the absolute Frobenius $F_B \colon B \to B$ and $F_L \colon L \to L$. Since $L/K$ is purely inseparable of degree $p$, $F_L$ factors through $K$. Hence, we find $F_B(B) \subseteq K \cap B = A$. Thus, $F_B^*$ factors through $\Cl(A) = 0$. Since $F_B^*$ is multiplication by $p$ on $\Cl(B)$, we get the result.
\end{proof}
\subsection{Covering involution and $ADE$ singularities}
Suppose now that $P \in X$ is a rational double point, that is, a normal, rational and Gorenstein surface singularity.
Equivalently, $P \in X$ is a canonical surface singularity.
Then, by adjunction, the exceptional divisors $E_i$ of a minimal resolution are smooth rational curves and 
satisfy $E_i^2=-2$. Their dual intersection graph is an $ADE$-Dynkin diagram, that is, $E$ is a root lattice with fundamental root system $\{E_1,\dots, E_n\}$. (See \cite[Section 1.4]{ebeling:lattices-and-codes} for the terminology and basic facts.)
The elements of the discriminant group $E^\vee/E$ are closely related to the exceptional divisors as follows:
Let $Z = \sum \alpha_i E_i$ be the so called fundamental cycle of the singularity, characterized among effective divisors by the property that its support is the whole exceptional locus and $Z.E_i \leq 0$ for all $i$. In Lie algebra terminology, the $E_i$ are called roots, and $Z$ is the \emph{highest root} of the positive root system generated by $E_1,\dots, E_n$ (cf. \cite[0.4.11]{cdl:enriquesI}). The number $\alpha_i$ is called the \emph{multiplicity} of the fundamental root $E_i$. The multiplicities of the fundamental roots are given in the following diagram (cf. \cite[0.4.12]{cdl:enriquesI}):
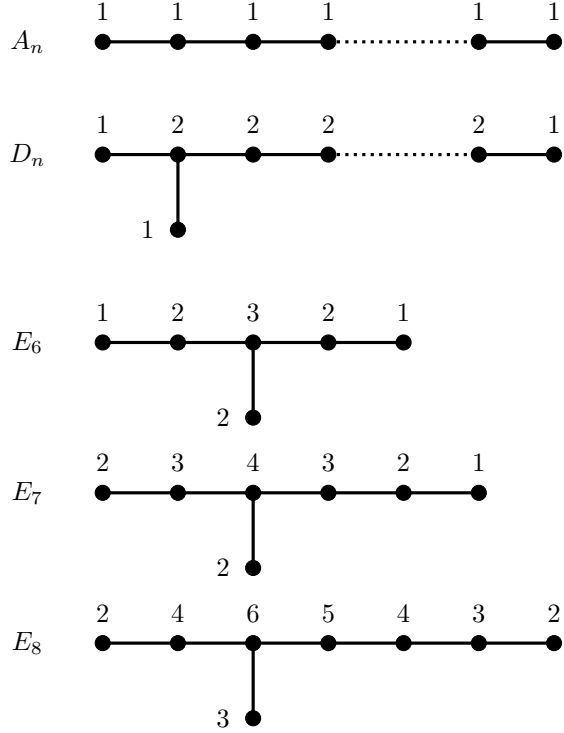
\begin{center}
\begin{tikzpicture}
\node at (0,0) {$A_n$};
\filldraw (1,0) circle (0.1);
\node at (1, 0.4) {$1$};
\filldraw (2,0) circle (0.1);
\node at (2, 0.4) {$1$};
\filldraw (3,0) circle (0.1);
\node at (3, 0.4) {$1$};
\filldraw (4,0) circle (0.1);
\node at (4, 0.4) {$1$};
\filldraw (6,0) circle (0.1);
\node at (6, 0.4) {$1$};
\filldraw (7,0) circle (0.1);
\node at (7, 0.4) {$1$};
\draw[line width=1.2pt]  (1,0) -- (4,0);
\draw[line width=1.2pt,dotted] (4,0) -- (6,0);
\draw[line width=1.2pt]  (6,0) -- (7,0);

\node at (0,-1.5) {$D_n$};
\filldraw (1,-1.5) circle (0.1);
\node at (1, -1.1) {$1$};
\filldraw (2,-1.5) circle (0.1);
\node at (2, -1.1) {$2$};
\filldraw (3,-1.5) circle (0.1);
\node at (3, -1.1) {$2$};
\filldraw (4,-1.5) circle (0.1);
\node at (4, -1.1) {$2$};
\filldraw (6,-1.5) circle (0.1);
\node at (6, -1.1) {$2$};
\filldraw (7,-1.5) circle (0.1);
\node at (7, -1.1) {$1$};
\filldraw (2,-2.5) circle (0.1);
\node at (1.6, -2.5) {$1$};
\draw[line width=1.2pt] (1,-1.5) -- (4,-1.5);
\draw[line width=1.2pt] (2,-1.5) -- (2,-2.5);
\draw[line width=1.2pt,dotted] (4,-1.5) -- (6,-1.5);
\draw[line width=1.2pt] (6,-1.5) -- (7,-1.5);
\node at (0, -4) {$E_6$};
\filldraw (1,-4) circle (0.1);
\node at (1, -3.6) {$1$};
\filldraw (2,-4) circle (0.1);
\node at (2, -3.6) {$2$};
\filldraw (3,-4) circle (0.1);
\node at (3, -3.6) {$3$};
\filldraw (4,-4) circle (0.1);
\node at (4, -3.6) {$2$};
\filldraw (5,-4) circle (0.1);
\node at (5, -3.6) {$1$};
\filldraw (3,-5) circle (0.1);
\node at (2.6, -5) {$2$};
\draw[line width=1.2pt] (1,-4) -- (4,-4);
\draw[line width=1.2pt] (3,-4) -- (3,-5);
\draw[line width=1.2pt] (4,-4) -- (5,-4);

\node at (0, -6) {$E_7$};
\filldraw (1,-6) circle (0.1);
\node at (1, -5.6) {$2$};
\filldraw (2,-6) circle (0.1);
\node at (2, -5.6) {$3$};
\filldraw (3,-6) circle (0.1);
\node at (3, -5.6) {$4$};
\filldraw (4,-6) circle (0.1);
\node at (4, -5.6) {$3$};
\filldraw (5,-6) circle (0.1);
\node at (5, -5.6) {$2$};
\filldraw (6,-6) circle (0.1);
\node at (6, -5.6) {$1$};
\filldraw (3,-7) circle (0.1);
\node at (2.6, -7) {$2$};
\draw[line width=1.2pt] (1,-6) -- (6,-6);
\draw[line width=1.2pt] (3,-6) -- (3,-7);
\end{tikzpicture}
\end{center}
\begin{center}
\begin{tikzpicture}
\node at (0, -8) {$E_8$};
\filldraw (1,-8) circle (0.1);
\node at (1, -7.6) {$2$};
\filldraw (2,-8) circle (0.1);
\node at (2, -7.6) {$4$};
\filldraw (3,-8) circle (0.1);
\node at (3, -7.6) {$6$};
\filldraw (4,-8) circle (0.1);
\node at (4, -7.6) {$5$};
\filldraw (5,-8) circle (0.1);
\node at (5, -7.6) {$4$};
\filldraw (6,-8) circle (0.1);
\node at (6, -7.6) {$3$};
\filldraw (7,-8) circle (0.1);
\node at (7, -7.6) {$2$};
\filldraw (3,-9) circle (0.1);
\node at (2.6, -9) {$3$};
\draw[line width=1.2pt] (1,-8) -- (7,-8);
\draw[line width=1.2pt] (3,-8) -- (3,-9);
\end{tikzpicture}
\end{center}
\begin{lemma}\label{prop:nonmultiple-roots-and-discriminant-group}
Let $E$ be an irreducible root lattice and $(e_1,\dots, e_n)$ a fundamental root system of $E$. 
Then, the roots of multiplicity $1$ are in bijection with the non-zero elements of the discriminant group $E^\vee/E$. Explicitly, let $e_1^\vee,\dots, e_n^\vee \in E^\vee$ with $e_i.e_j^\vee = \delta_{ij}$ be the dual basis.
Suppose that $e_i$ is of multiplicity $1$. Then, the bijection maps $e_i$ to $e_i^\vee+E \in E^\vee/ E$.
\end{lemma}
\begin{proof}
This follows from the statement and proof of \cite[Lemma 1.18]{ebeling:lattices-and-codes}.
\end{proof}
The following Lemma in the case of characteristic $p \neq 2$ is due to Shimada \cite{shimada:automorphisms_of_supersingular_k3s_and_salem} (see also \cite[Lemma 8.7.20]{enriquesII}). Shimada's proof is by direct calculation and the classification of $ADE$ singularities and their automorphisms of order $2$. 
We are now in a position to give a conceptual proof in any characteristic.
\begin{lemma}\label{prop:action-on-exceptional}
Let $\phi \colon X \to Y$ be a finite separable morphism of degree 2 between normal proper surfaces. 
Suppose that $P \in X$ is a rational double point whose image $\phi(P)$ is a smooth point of $Y$. Let $\phi'\colon \widetilde{X} \to X$ be the minimal resolution of $P$. Let $E$ be the subgroup of $\Num(\widetilde{X})$ spanned by exceptional divisors.

Then, the deck transformation $\tau \in \Aut(X/Y)$ extends to an automorphism $\widetilde{\tau} \in \Aut(\widetilde{X}/Y)$, and acts on the 
discriminant group $E^\vee/E$ as $-1$. Therefore, on the
dual graph of the exceptional locus, $\widetilde{\tau}$ acts as follows
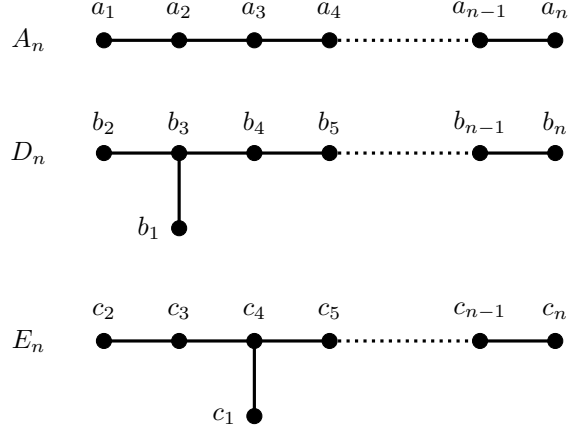
\begin{center}
\begin{tikzpicture}
\node at (0,0) {$A_n$};
\filldraw (1,0) circle (0.1);
\node at (1, 0.4) {$a_1$};
\filldraw (2,0) circle (0.1);
\node at (2, 0.4) {$a_2$};
\filldraw (3,0) circle (0.1);
\node at (3, 0.4) {$a_3$};
\filldraw (4,0) circle (0.1);
\node at (4, 0.4) {$a_4$};
\filldraw (6,0) circle (0.1);
\node at (6, 0.4) {$a_{n-1}$};
\filldraw (7,0) circle (0.1);
\node at (7, 0.4) {$a_n$};
\draw[line width=1.2pt]  (1,0) -- (4,0);
\draw[line width=1.2pt,dotted] (4,0) -- (6,0);
\draw[line width=1.2pt]  (6,0) -- (7,0);

\node at (0,-1.5) {$D_n$};
\filldraw (1,-1.5) circle (0.1);
\node at (1, -1.1) {$b_2$};
\filldraw (2,-1.5) circle (0.1);
\node at (2, -1.1) {$b_3$};
\filldraw (3,-1.5) circle (0.1);
\node at (3, -1.1) {$b_4$};
\filldraw (4,-1.5) circle (0.1);
\node at (4, -1.1) {$b_5$};
\filldraw (6,-1.5) circle (0.1);
\node at (6, -1.1) {$b_{n-1}$};
\filldraw (7,-1.5) circle (0.1);
\node at (7, -1.1) {$b_n$};
\filldraw (2,-2.5) circle (0.1);
\node at (1.6, -2.5) {$b_1$};
\draw[line width=1.2pt] (1,-1.5) -- (4,-1.5);
\draw[line width=1.2pt] (2,-1.5) -- (2,-2.5);
\draw[line width=1.2pt,dotted] (4,-1.5) -- (6,-1.5);
\draw[line width=1.2pt] (6,-1.5) -- (7,-1.5);

\node at (0, -4) {$E_n$};
\filldraw (1,-4) circle (0.1);
\node at (1, -3.6) {$c_2$};
\filldraw (2,-4) circle (0.1);
\node at (2, -3.6) {$c_3$};
\filldraw (3,-4) circle (0.1);
\node at (3, -3.6) {$c_4$};
\filldraw (4,-4) circle (0.1);
\node at (4, -3.6) {$c_5$};
\filldraw (6,-4) circle (0.1);
\node at (6, -3.6) {$c_{n-1}$};
\filldraw (7,-4) circle (0.1);
\node at (7, -3.6) {$c_n$};
\filldraw (3,-5) circle (0.1);
\node at (2.6, -5) {$c_1$};
\draw[line width=1.2pt] (1,-4) -- (4,-4);
\draw[line width=1.2pt] (3,-4) -- (3,-5);
\draw[line width=1.2pt,dotted] (4,-4) -- (6,-4);
\draw[line width=1.2pt] (6,-4) -- (7,-4);
\end{tikzpicture}
\end{center}
\begin{itemize}
 \item $\widetilde{\tau}(a_i)=a_{n+1-i}$, $i=1,\dots,n$;
 \item \(\widetilde{\tau}(b_i)=b_i\) if $n$ is even;
 \item $\widetilde{\tau}(b_1)=b_2$, $\widetilde{\tau}(b_2)=b_1$, $\widetilde{\tau}(b_i)=b_i$, $i\neq 1,2$ if $n$ is odd;
 \item $\widetilde{\tau}(c_1)=c_1$, $\widetilde{\tau}(c_i)=c_{8-i}$, $i\neq 1$ if $n=6$;
 \item $\widetilde{\tau}(c_i)=c_i$ if $n=7,8$.
\end{itemize}
\end{lemma}
\begin{proof}
By uniqueness of the minimal resolution of a surface singularity, $\tau$ extends to an automorphism $\widetilde{\tau} \in  \Aut(\widetilde{X}/Y)$. As $P$ is a singular point and $\phi(P)$ is not, $P$ is fixed by $\tau$.

Since the rest of the statement is local, we may pass to the completion at $P$ (and its image). 
By \Cref{separable-galois-cover}, $\tau$ acts on the local class group $\Cl_P(X)$ as $-1$, and by \Cref{class-group_and_discriminant-group} there is a natural isomorphism $\Cl_P(X)\cong E^\vee/E$. 
Finally, the effective curves in $E$ form a positive root system $\Phi^+$ whose fundamental roots are the irreducible curves in $E$ (that is, the $a_i$, $b_i$ or $c_i$). By \Cref{prop:nonmultiple-roots-and-discriminant-group}, the nonmultiple roots $s_i$ of $E$ are in bijection with the non-zero elements of the discriminant group $E^\vee / E$ via $s_i \mapsto [s_i^\vee]$.
By an explicit computation of the discriminant group, one sees that
\begin{itemize}
\item for $A_n$, all the fundamental roots are nonmultiple and 
$-[a_i^\vee]=[a_{n+1-i}^\vee]$. 
\item for $D_n$, $b_1,b_2,b_n$ are nonmultiple; 
\begin{itemize}
\item for $n$ odd, $-[b_1^\vee]=[b_2^\vee]$ and $2[b_1^\vee]=[b_n^\vee]$; 
\item for $n$ even, $-[b_i^\vee]=[b_i^\vee]$ for $i=1,2,n$;
\end{itemize}
\item for $E_6$, $c_2,c_6$ are nonmultiple, $-[c_2^\vee]=[c_6^\vee]$ 
\item for $E_7$, only $c_7$ is nonmultiple;
\item there is no nonmultiple root in $E_8$.
\end{itemize}
This gives the claimed action on the resolution graph. 
\end{proof}
By the following  well-known lemma, every generically finite morphism from a surface with numerically trivial canonical bundle factors through a finite morphism from a model with rational double points.
\begin{lemma}\label{has-rdp}
Let $\pi \colon \widetilde{X} \to Y$ be a generically finite morphism between proper surfaces. Assume that $\widetilde{X}$ is smooth and consider the Stein factorisation 
\[\widetilde{X} \xrightarrow{\phi'} X \xrightarrow{\phi} Y.\]
If $K_{\widetilde{X}}.D \leq 0$ for all curves $D$ on $\widetilde{X}$ contracted by $\phi'$, then $X$ has at most rational double point singularities.
\end{lemma}
\begin{proof}
Any irreducible curve $D$ contracted by $\phi'$ has negative self-intersection $D^2<0$ by \cite[p.186]{lipman78}. By assumption, we have
\[K_{\widetilde{X}}.D \leq 0.\] 
By adjunction, it follows that
\[2p_a(D) - 2 = D^2 + K_{\widetilde{X}}.D \leq D^2.\] 
From $D^2<0$, we conclude that $p_a(D)=0$, so that $D$ is a smooth rational curve with $D^2 \in \{-1,-2\}$.
Replacing $\widetilde{X}$ by the (iterated) contraction of all $D$ with $D^2 = -1$ and using that $K_{\tilde{X}}.D$ does not increase under such a contraction, we may assume that $D^2 = -2$ for all exceptional curves $D$. Then, $\widetilde{X} \to X$ is a minimal resolution of singularities. 
We conclude with the following statement: A surface singularity is a rational double point if and only if all the exceptional curves in a minimal resolution are smooth rational curves of self-intersection $-2$ (cf. \cite[Theorem 3.31]{buadecu01}). 
\end{proof}
\section{The cone conjecture via bielliptic double covers}
The following proposition can be found in \cite{allcock2018} and \cite[Proposition 8.3.2]{enriquesII} and is attributed to Coble. 
\begin{proposition}\label{prop:U-pairs-generate-G0}
The $2$-congruence subgroup
\[O^+(E_{10})(2)\defeq \ker \left(O^+(E_{10}) \to O(E_{10} \otimes \FF_2)\right)\] 
coincides with
\[\langle \id_U \oplus -\id_{U^\perp} \mid U \subseteq E_{10} \mbox{ is a hyperbolic plane} \rangle.\] 
\end{proposition}
\begin{lemma}\label{lem:-1-in-weyl-group}
Let $E$ be an irreducible root lattice and $g \in O(E)$. Then, \[g \in-1W(E) \subseteq O(E)\] 
if and only if 
$g$ acts as $-\id_{E^\vee/E}$ on the discriminant group $E^\vee/E$.
\end{lemma}
\begin{proof}
Pick any fundamental root system $\Gamma$. Then, it holds that
\[O(E) = W(E) \rtimes \mathrm{Sym}(\Gamma),\]
where $\mathrm{Sym}(\Gamma)$ consists of graph automorphisms of the Coxeter-Dynkin diagram of $\Gamma$. 
Since $E^\vee/E\setminus\{0\}$ is in bijection with the nonmultiple roots of $\Gamma$ by \Cref{prop:nonmultiple-roots-and-discriminant-group}, $\mathrm{Sym}(\Gamma)$ acts faithfully on $E^\vee/E$.
Thus, the natural homomorphism
\[O(E) \to O(E^\vee/E)\]
has kernel $W(E)$, and $O(E)/W(E) \cong \mathrm{Sym}(\Gamma)$ injects into $O(E^\vee/E)$. The claim follows because $-1 \in O(E)$ maps to $-\id_{E^\vee/E}$.
\end{proof}
\begin{proposition}\label{prop:bielliptic-and-weyl-group}
Let $(f_1,f_2)$ be a possibly degenerate $U$-pair on the Enriques surface $S$ and let $\phi\colon S \to D$ be the corresponding bielliptic map.
Then, we have
\[(\id_U \oplus -\id_{U^\perp}) \in W(S)\rtimes\Autstar(S)=\vcentcolon G_S,\]
where $\rtimes$ denotes the inner semidirect product and $U=\langle f_1,f_2\rangle_\mathbb{Z}$.
\end{proposition}
\begin{proof}
The proposition is implied by the following two claims:
\begin{enumerate}
\item If $\phi$ is \emph{separable}, and $\tilde \tau\in \Aut(S)$ is the corresponding bielliptic involution, then 
\[\tilde \tau^*\cdot W(S) = (\id_U \oplus -\id_{U^\perp})\cdot W(S).\] 
\item If $\phi$ is inseparable, then 
\[(\id_U \oplus -\id_{U^\perp}) \in W(S).\]
\end{enumerate}
Let $|L|$ with $\phi=\phi_{|L|}$ denote the corresponding bielliptic linear system. If $|L|$ is non-special, we let $F_1,F_2$ be half-fibers with $|L| = |2F_1 + 2F_2|$, and if $L$ is special, we let $F$ be a half-fiber and $R$ a smooth rational curve with $|L| = |4F + 2R|$.
Let 
\[S \xrightarrow{\phi_1} S' \xrightarrow{\phi_2} D\] 
be the Stein factorization of $\phi$.
The morphism $\phi_2$ is finite of degree $2$ and
the morphism $\phi_1$ is birational, has connected fibers, and contracts the $(-2)$-curves orthogonal to $L$ to rational double point singularities on $S'$ by \Cref{has-rdp}.
If $L$ is non-special, these are the ones in $U^\perp$, and if it is special, the bisection $R$ is contracted as well. 
By \cite[Lemma 3.3.6]{cdl:enriquesI}, its image $q=\phi(R)$ is an $A_1$ singulartiy - the only singular point of $D$. Let $r$ denote the class of $R$.
Let $e_{ij}\neq r$ denote the classes of the $(-2)$-curves contracted to $p_i \in S'$ and let $E_i =\langle e_{ij} | j \rangle$ and $E=\langle e_{ij} \mid i,j\rangle $ be the sublattices they generate. 

(1) Suppose that $\phi$ is separable. The involution $\tilde \tau$ permutes the $e_{ij}$ with $i$ fixed because they map to the same point $p_i$. The half--fibers $F_1$ and $F_2$ (resp. $F$ in the special case) map to lines in $D$, hence, they are preserved by $\tilde \tau$. In the special case, the rational curve $R$ is contracted to a node in $S'$, since every other smooth rational curve $R'$ on $S$ with $R'.L = 0$ must satisfy $F.R' = R.R' = 0$. Thus, in any case, $\tilde \tau$ fixes the hyperbolic plane $U$ determined by the bielliptic map $\phi$. Since the Picard number of $D$ is $1$ in the non-special and $2$ in the special case, this shows that the pullback of $\Num(D)_\QQ$ is contained in $U_\QQ$.
Note that $\phi_2^*\phi_{2*} x = x + \tau^{*}x$ where $\tau \in \Aut(S')$ is the deck transformation. Hence the fixed space $\Num(S')^{\tau}_\QQ = \phi_2^*(\Num(D)_\QQ)$.  We see that $\tilde \tau^*$ acts as $-1$ on $(U\oplus E)^\perp$. 

To complete the proof for the separable case, it suffices to show that
\[\tau^*|_{E_i} \cdot W(E_i) =-1W(E_i) \mbox{ for all } i. \]
Since $\phi_2(p_i)\neq q$ is a smooth point of $D$, the claim follows from \Cref{prop:action-on-exceptional,has-rdp}. 

(2) Now, suppose that $\phi$ is inseparable. By \cite[Proposition 3.3.13]{cdl:enriquesI}, $\phi$ contracts $8$ smooth rational curves if $D$ is non-special, and $9$ if it is special (the extra one is $R$). Hence, $U^\perp _\QQ = E _\QQ$ is spanned by $8$ smooth rational curves contracted by $\phi$.
By \Cref{prop:local-class-inseparable}, the local class group $\Cl_{p_i}(S')$ is $2$-torsion. By \Cref{class-group_and_discriminant-group}, $\Cl_{p_i}(S')\cong E_i^\vee / E_i$, and, by \Cref{lem:-1-in-weyl-group} applied with $g=1$, we have $-\id_{E_i}\oplus \id_{E_i^\perp} \in W(E_i)$, where we consider $W(E_i)\subseteq W(S)$ via the inner semidirect product. 
Since $U^\perp_\QQ=E _\QQ = \left(\bigoplus_i E_i\right)_\QQ$, and $U _\QQ = E^\perp _\QQ$ this implies 
\[\id_U \oplus -\id_{U^\perp} = \id_{E^\perp}\oplus -\id_E  \in W(E) \subseteq W(S).\] 
\end{proof}
\begin{theorem}\label{prop:2-congruence}
Let $S$ be an Enriques surface. Then the subgroup of the Weyl automorphism group $G_S$ generated by $W(S)$ together with all bielliptic involutions,
contains the $2$-congruence subgroup $O^+(\Num(S))(2)$.
\end{theorem}
\begin{proof}
Let $f_1,f_2 \in \Num(S)$ span a hyperbolic plane $U$, and set $l=2f_1+2f_2$. By \cite[Proposition 6.1.5]{enriquesII}, after applying an element $w$ of the Weyl group, we may assume that $f_1$ and $l$ are nef and the linear system $|L|$ of $l$ is, therefore, bielliptic. 
By \Cref{prop:bielliptic-and-weyl-group}, we have
\[\id_U \oplus -\id_{U^\perp} \in \tau W(S) \subseteq G_S\]
where $\tau$ is either the identity or the corresponding bielliptic involution.
This shows that the group generated by $W(S)$ and the bielliptic involutions, contains $\id_{U}\oplus -\id_{U^\perp}$. 
Since $U$ was arbitrary, \Cref{prop:U-pairs-generate-G0} yields the claim. 
\end{proof}
\begin{corollary}
$\Aut(S)$ admits a finite index subgroup $B(S)$, which is generated by bielliptic involutions. 
\end{corollary}
\begin{proof}
The kernel of $\Aut(S) \to \Autstar(S)$ consists of so called numerically-trivial automorphisms. By \cite[Proposition 2.1]{dolgachev_martin:numericallytrivial} it is finite. Therefore, it suffices to show that $\Autstar(S)$ is generated by bielliptic involutions up to finite index. 
Let $B(S) \subseteq \Autstar(S)$ be the subgroup generated by bielliptic involutions. 
Let $f\in O^+(\Num(S))(2)\cap\Autstar(S)$, which is of finite index in $\Autstar(S)$. 
By \Cref{prop:bielliptic-and-weyl-group}, it holds that $f = b w$ with $b \in B(S)$ and $w \in W(S)$. Since $W(S)\cap\Aut^*(S) =1$, we conclude that $w=1$.  
\end{proof}
\begin{question}
Which numbers can occur as $[\Aut(S):B(S)]$ when $S$ is any Enriques surface?
\end{question}
\begin{corollary}\label{coro:core-result}
The Morrison--Kawamata cone conjecture holds for Enriques surfaces over algebraically closed fields, that is $\Autstar(S)$ has a rational polyhedral fundamental domain on $\Nef^e(S)$.
\end{corollary}
\begin{proof}
We have $\Nef^e(S)=\Nef^+(S)$. The fundamental domain of $\Autstar(S)$ on $\Nef^+(S)$ is a fundamental domain for the action of $G_S$ on $\P^+_S$. 
It follows with \cite[Definition-Proposition 4.1]{looijenga2014} that the fundamental fan can be chosen to be rational polyhedral, because $G_S$ is of finite index in $O^+(\Num(S))$ and 
the fundamental domain of $O^+(\Num(S))=W(\Num(S))$ on $\P^+_S$ is rational polyhedral.
Let $e_1,\dots, e_{10} \in E_{10}$ be roots which form an $E_{10}$ configuration. 
The fundamental chamber 
\[\{x \in E_{10}\otimes \RR \mid x.e_i\geq 0, i=1,\dots,10 \}\] 
of $W(\Num(S))$ is named in honor of E.B. Vinberg who calculated it in \cite{vinberg1972}.
It is obviously rational polyhedral.
\end{proof}
\begin{proof}[Proof of \Cref{thm: main}]
This follows with \Cref{rem:perfect-field} and \Cref{coro:core-result}.
\end{proof}
A consequence of the Morrison--Kawamata cone conjecture is that $S$ admits only finitely many $\Aut(S)$-orbits of (numerical classes of) nef line bundles of a given square. For Enriques surfaces we even have precise counting formulas. For example for nef line bundles of square zero, i.e., genus $1$-fibrations. 
Let $\varphi \colon S \to \mathbb{P}^1$ be a genus $1$-fibration and $f$ the class of a half fiber. 
We call the integer 
\[r_\varphi:= \sharp \overline{G_S \cdot f}=\sharp \{\overline{g(f)} \in \Num(S)\otimes \FF_2 \mid g \in G_S \}\]
the ramification degree of $\varphi$. 
If $S$ is a complex Enriques surface, then $r_\varphi$ can be interpreted as the ramification degree of the forgetful map 
\[\mathcal{M}_{En,e} \to \mathcal{M}_{En},\quad [S, \varphi] \to [S]\]
from the moduli space of elliptic Enriques surfaces $\mathcal{M}_{En,e}$ to the moduli space of Enriques surfaces $\mathcal{M}_{En}$(see \cite{brandhorst-gonzalez:527}). 
By results of Barth--Peters the degree of this map is $527$ \cite[Section 0]{barth-peters:very_general_enriques}. 
\begin{corollary}\label{coro:527}
$S$ has exactly $527$ genus $1$-fibrations, when counted with their ramification degree.
\end{corollary}
\begin{proof}
Since $O^+(\Num(S))(2)\subseteq G_S$, the main result of \cite{brandhorst-gonzalez:527} extends to characteristic $2$ and yields the claim.
\end{proof}
We now return to the situation where $k$ is perfect field which is not necessarily algebraically closed.  
Let $\HH^{\rho_S-1}$ denote one of the two sheets of $\{x \in \Num(S)_\mathbb{R} \mid x^2 = 1\}$. It is a $\rho_S-1$-dimensional hyperbolic space. Recall from \cite[Definition 2.3]{kikuta2024} that a representation $G \to {\rm Isom}(\mathbb{H}^n)$ is called \emph{geometrically finite} if it has finite kernel and its image is finitely generated with finite covolume in an $\epsilon$-neighborhood of the convex hull of its limit set. The group $G$ is called $\mathrm {CAT}(0)$ if it admits a geometrically proper, cocompact, and isometric action on a $\mathrm {CAT}(0)$-space. This has far reaching implications, for instance the word and the conjugacy problem for $G$ are solvable (cf. \cite[Chapter III.$\Gamma$ Theorems 1.1 and 1.2]{bridson1999}).
\begin{corollary}\label{coro:geometrically-finite}
 The natural representation $\Aut(S) \to \mathrm{Isom}(\HH^{\rho_S-1})$ is geometrically finite. In particular, $\Aut(S)$ is $\mathrm {CAT}(0)$ and $\Aut(S)$ is finitely presented.
\end{corollary}
\begin{proof}[Proof of \Cref{coro:of-main-thm,coro:geometrically-finite}]
This is a consequence of the cone conjecture for Enriques surfaces and the finiteness of the group of numerically trivial automorphisms \cite[Proposition 2.1]{dolgachev_martin:numericallytrivial} as explained by Kikuta \cite[Theorem A and Corollary 4.3]{kikuta2024}. 
\end{proof} 
\printbibliography
\end{document}